\documentclass{amsart}
\usepackage{amssymb}
\usepackage{eucal}
\theoremstyle{plain}
\newtheorem{theorem}{Theorem}
\newtheorem{corollary}{Corollary}

\newtheorem{lemma}{Lemma}

\theoremstyle{definition}

\theoremstyle{remark}

\numberwithin{equation}{section}
\newenvironment{theorem*}[1][Theorem.]{\begin{trivlist}
\item[\hskip \labelsep {\bfseries #1}]}{\end{trivlist}}

\newcommand{\R}{\mathbb R}

\newcommand{\Z}{\mathbb Z}
\newcommand{\C}{\mathbb C}

\newcommand{\Rn}{\mathbb R^n}

\begin{document}
\title[Unique continuation]{A Theorem of Paley-Wiener Type for  Schr\"odinger Evolutions}
\author{C. E. Kenig}
\address[C. E. Kenig]{Department of Mathematics\\University of Chicago\\Chicago, Il.
60637 \\USA.}
\email{cek@math.uchicago.edu}
\thanks{The first and second  authors are supported  by NSF grants DMS-0968472 and
DMS-1101499 respectively. The third author is supported by MEC grant
MTM2004-03029}
\author{G. Ponce}
\address[G. Ponce]{Department of Mathematics\\ University of California\\ Santa
Barbara, CA 93106\\ USA.}
\email{ponce@math.ucsb.edu}
\author{L. Vega}
\address[L. Vega]{UPV/EHU\\Dpto. de Matem\'aticas\\Apto. 644, 48080 Bilbao, Spain.}
\email{luis.vega@ehu.es}
\begin{abstract}
We prove unique continuation principles for solutions of  evolution Schr\"odinger equations with 
time dependent potentials. These  correspond  to uncertainly principles of Paley-Wiener type for the  
Fourier transform. Our results extends to a large class of semi-linear
Schr\"odinger equation. 
\end{abstract}
\maketitle


\section{Introduction}\label{S: Introduction}

In this paper we  study unique continuation properties of
solutions of Schr\"odinger equations of the form
\begin{equation}
\label{E:1.1}
\partial_t u = i( \triangle u + V(x,t) u),\;\;\;\;\;\; (x,t)\in \R^n\times
[0,T], \;\;\;\;\;\;T>0.
\end{equation}

The goal is to obtain sufficient conditions on the behavior of the
solution $u$
at two different times and on the potential $V$ which guarantee that
$u\equiv 0$ in $\Rn\times [0,T]$.
Under appropriate assumptions this result will extend to the difference $v=u_1-u_2$ of two solutions $u_1,\;u_2$ of
semi-linear
Schr\"odinger equation
\begin{equation}
\label{E: NLS}
\partial_t u =i( \triangle u + F(u,\overline u)),
\end{equation}
from which one can conclude that $u_1\equiv u_2$.
\vskip.1in

Defining the Fourier transform of a function $f$ as 
$$
\widehat f(\xi)=(2\pi)^{-n/2} \int_{\R^n} e^{-i\xi\cdot x} f(x)dx.
$$
one has 
\begin{equation}
 \label{formula-1}
\begin{aligned}
u(x,t)&= e^{it\Delta}u_0(x)= \int_{\R^n} \frac{e^{i|x-y|^2/4t}}{(4\pi i t)^{n/2}}\,
u_0(y)\,dy\\
\\
&=\frac{e^{i|x|^2/4t}}{(4\pi i t)^{n/2}} \int_{\R^n}e^{-2ix\cdot y/4t} e^{i|y|^2/4t}
u_0(y)\,dy\\
\\
&= \frac{e^{i|x|^2/4t}}{(2 i t)^{n/2}}\;
\widehat{\;(e^{i|\cdot|^2/4t}u_0)\,}\left(\frac{x}{2 t}\right),
\end{aligned}
\end{equation}
where $e^{it\Delta}u_0(x)$ denotes  the 
free solution of the
Schr\"odinger equation with data $u_0$
$$
\partial_tu=i \triangle u,\;\;\;\;\;\;u(x,0)=u_0(x),\;\;\;\;(x,t)\in\Rn\times \R.
$$
The identity \eqref{formula-1} tells us that this kind of results for the free solution of the
Schr\"odinger equation are closely related to uncertainty principles for the Fourier transform. In 
this regard, one has the well
known result of G. H. Hardy \cite{Hardy}:
$$
\aligned
&\text{If}\;\;f(x)=O(e^{-x^2/\beta^2}), \;\;\widehat f(\xi)=O(e^{-4\,\xi^2/\alpha^2})\;\;\text{and}\;\;
\alpha\beta<4,\;\;\text{ then}\;\; f\equiv 0,\\
& \;\;\text{and}\;\;\text{if}\;\; \alpha\beta=4,\;\;\text{then}\;\;f(x)=c\,e^{-x^2/\beta^2}.
 \endaligned
 $$ 
Its extension to higher dimensions $n\geq 2$ was obtained  in \cite{SST}. 
The following generalization in terms of the $L^2$-norm  was established  in \cite{CoPr}:
$$
\text{If}\;\;\;\;e^{\frac{|x|^2}{\beta^2}}\,f(x),\;\;e^{\frac{4|\xi|^2}{\alpha^2}}\widehat
f(\xi)\in L^2(\mathbb R^n),\;\;\text{and}\;\;\alpha\,\beta\leq
4,\;\text{then}\;f\equiv 0.
$$

In terms of the free solution of the Schr\"odinger equation the
$L^2$-version of Hardy Uncertainty Principle
says :
\begin{equation}
\label{hardy-l2}
\text{If}\;\;
\;e^{\frac{|x|^2}{\beta^2}}\,u_0(x),\;\;e^{\frac{|x|^2}{\alpha^2}}\,e^{it\Delta}\,u_0(x)\in
L^2(\mathbb R^n),\;\;\text{and}\;\;\alpha\,\beta\leq
4t,\;\text{then}\;u_0\equiv 0.
\end{equation}

In \cite{EKPV09}  the following result was proven:
\vskip.1in

\begin{theorem*}(\cite{EKPV09})
Given  any solution  $u \in C([0,T] :L^2(\Rn))$ of
\begin{equation}
\label{ivp11}
\partial_tu=i\left(\triangle u+V(x,t)u\right),\,\;\;\;\;\;\,\;\;\;(x,t)\in \Rn\times [0,T],
\end{equation}
with  $V\in L^{\infty}(\R^n\times [0,T])$,
\begin{equation}
\label{14a}
\lim_{\rho\rightarrow +\infty}\|V\|_{L^1([0,T] : L^\infty(\R^n\setminus B_{\rho}))}=0.
\end{equation}
and
$$
e^{\frac{|x|^2}{\beta^2}}\,u_0,\;\;\;e^{\frac{|x|^2}{\alpha^2}}\,e^{i T \Delta}u_0\in
L^2(\mathbb R^n),
$$
with $\alpha\,\beta<4 T$,
then $\,u_0\equiv 0$.
\end{theorem*}
\vskip.1in

Notice that the above Theorem recovers the $L^2$-version of the Hardy
Uncertainty Principle \eqref{hardy-l2} for solutions of the IVP
\eqref{ivp11},
except for the limiting case $\alpha\,\beta=4T$  for which the
corresponding result was proven to fail, see  \cite{EKPV09}.
Also one has  the uncertainty principle of the type first studied by  G. W. Morgan in \cite{Mo}. 
The following result was proven in  \cite{Ho} for the one dimensional case $n=1$ and extended to higher 
dimension in \cite{bonamie2} and  \cite{ray}: 
$$
\text{If} \;\;f\in L^1(\mathbb R^n)\cap L^2(\mathbb R^n)\;\;\;\text{and}\;\;\int_{\mathbb R^n} \int_{\mathbb R^n} |f(x)| |\widehat f(\xi)| e^{
|x\,\cdot \xi|}\,dx\,d\xi<\infty, \;\;\text{then}
\;\;f\equiv 0.
$$

In particular, using Young's inequality  this  implies :
\vskip.05in
If 
$f\in L^1(\mathbb R^n)\cap L^2(\mathbb R^n)$, $\;p\in(1,2)$, $\,1/p+1/q=1$, $\,\alpha,\,\beta>0$, and 
\begin{equation}
\label{primera}
 \int_{\mathbb R^n}|f(x)|\, e^{\frac{\alpha^p|x|^p}{p}}dx \,+
\,\int_{\mathbb R^n} |\widehat
f(\xi)| \,e^{\frac{\beta^q|\xi|^q}{q}}d\xi<\infty,\;\;\alpha \beta\geq 1\;\Rightarrow \;f\equiv 0,
 \end{equation}
 or in terms of the solution of the free Schr\"odinger
equation :
\vskip.05in
If $u_0\in L^1(\mathbb R^n)\cap L^2(\mathbb R^n)$  and for
some $\,t\neq 0$
 \begin{equation}
 \label{pq}
 \int_{\mathbb R^n}\;|u_0(x)|\, e^{\frac{\alpha^p|x|^p}{p}}dx \,+
\,\int_{\mathbb R^n}\,
|\,e^{it\Delta}u_0(x)|
\,e^{\frac{\beta^q|x|^q}{q(2t)^q}}dx<\infty,\;\;\;
\alpha\,\beta\geq 1,
\end{equation}
then $\;u_0\equiv 0$. 

 In the one dimensional case $n=1$ the above condition  $\alpha\beta\geq 1$ 
 can be replaced by an optimal one  $\;\alpha\,\beta>|\cos(p\pi/2)|^{1/p}\;$ as was established in   \cite{Mo}, \cite{monki}, \cite{bonamie2}.
 \vskip.05in
 In \cite{EKPV08m}  the following result was obtained :
\vskip.1in
\begin{theorem*} (\cite{EKPV08m})
Given $\,p\in(1,2)$ there exists $N_p>0$ such that if
\newline $u\in C([0,1]:L^2(\mathbb R^n))$ is a solution of
$$
\partial_t u=i (\Delta u +V(x,t)u),\;\;\;\;\;\;\;\;\;\;(x,t)\in\R^n\times[0,1],
$$
such that  $V\in L^{\infty}(\R^n\times [0,1])$,
$$
\lim_{\rho\rightarrow +\infty}\|V\|_{L^1([0,1] : L^\infty(\R^n\setminus B_{\rho}))}=0,
$$
and there exist $\,\alpha,\,\beta>0$
\begin{equation}
\label{con1}
\int_{\mathbb R^n}
|u(x,0)|^2e^{2\,\alpha^p\,|x|^p/p}dx\;\,\,+\,\,\;\int_{\mathbb
R^n}
|u(x,1)|^2e^{2\,\beta^q\,|x|^q/q}dx<\infty,
\end{equation}
with $\,1/p+1/q=1$ and 
\begin{equation}
\label{conditionp2}
\;\alpha\,\beta \geq N_p,
\end{equation}
then $\;u\equiv 0$.
\end{theorem*}

 Some previous results concerning uniqueness
properties of solutions
of the Schr\"odinger equation  were not directly motivated by the formula 
\eqref{formula-1}.

 For solutions $u=u(x,t)$ of the $1$-D cubic Schr\"odinger equation
 \begin{equation}
 \label{cubicNLS}
 \partial_t u=i(\partial_x^2 u\pm |u|^2u),
\end{equation}
 B. Y. Zhang \cite{BZ} showed :
 \vskip.05in 
  If  $u(x,t)=0$ for $(x,t)\in (-\infty, a)\times \{0,1\}\,\, ($or $(x,t)\in
(a,\infty)\times \{0,1\})\,$ for some 
$\,a\in \R$, then $u\equiv 0$.
\vskip.03in
The proof is based on the inverse scattering method, which uses the fact that the
equation in \eqref{cubicNLS} is a completely integrable model. 
\vskip.05in
In \cite{KPV02}, under general assumptions on $F$ in \eqref{E: NLS}, it was proven that :
 \vskip.05in
If $u_1,\,u_2\in C([0,1]:H^s(\R^n))$, with $\,s>\max \{n/2; \,2\}\, $ are
solutions of the equation \eqref{E: NLS} with
$F$ as in \eqref{E: NLS} such that
$$
u_1(x,t)=u_2(x,t),\;\;\;(x,t)\in \Gamma^c_{x_0}\times \{0,1\},
$$
where $ \Gamma^c_{x_0}$ denotes the complement of a cone  $\Gamma_{x_0}$ with vertex
$x_0\in \R^n$ and opening $<180^0$, 
then
$u_1\equiv u_2$.

For further results in this direction see \cite{Iza}, \cite{KPV02}, \cite{IK04}, \cite{IK06},
and 
references therein.  Note that in \cite{EKPV12} a unified approach was given to both kinds of results, using 
Lemma \ref{Lemma3} and Corollary \ref{Corollary1} below.
\vskip.05in

Returning to the uncertainty principle for the Fourier transform one has :
 \vskip.05in
 If $f\in L^1(\R^n)$ is non-zero and has compact support, then
$\,\widehat f $
cannot satisfy a condition of the type
$\widehat f(y)=O(e^{-\epsilon |y|})$ for any $\epsilon>0$. 
 \vskip.05in
This is due to the fact that $\widehat f(y)=O(e^{-\epsilon |y|})$ implies that $f$ has an analytic extension 
to the strip $\{z\in\C^n\,:\,| Im (z)|<\epsilon\}$. 
However, it may be possible to have  $ f\in L^1(\R^n)$ a non-zero function with
compact support, such that
$\widehat f(\xi)=O(e^{-\epsilon(y) |y|})$, $\epsilon(y)$ being a positive function
tending to zero as 
$|y|\to \infty$. In the one-dimensional case ($n=1$) A. E.
Ingham \cite{In} 
proved the following :
 \vskip.05in
There exists $f\in L^1(\R)$ non-zero, even, vanishing outside an interval such
that
$\widehat f(y)=O(e^{-\epsilon(y) |y|})$ with $\epsilon(y)$ being a positive
function tending to zero at infinity
if and only if
$$
\int^{\infty}_a \frac{\epsilon(y)}{y}\,dy<\infty,\;\;\;\;\;\;\;\;\;\text{for some}\;\;\;\;\;\;a>0.
$$
\vskip.1in

 In this regard the  Paley-Wiener Theorem \cite{PW} gives a characterization of
a function 
or distribution with compact support in term of the analyticity properties of its
Fourier transform.

 Our main result in this work is the following:
 
\begin{theorem}\label{Theorem1} Let  
 $u\in C([0,1]:L^2(\mathbb R^n))$ be a strong solution of the equation
 \begin{equation}
 \label{maineq}
 \partial_t u= i(\Delta u+V(x,t)u),\;\;\;\;\;\;\;(x,t)\in\R^n\times [0,1].
 \end{equation}
 Assume that
 \begin{equation}
 \label{cond1}
 \sup_{0\leq t\leq 1}\,\int_{\R^{n}}|u(x,t)|^2dx\leq A_1,
 \end{equation}
 \begin{equation}
\label{cond2}
\int_{\R^n}\,e^{2a_1|x_1|}\,|u(x,0)|^2\,dx\leq A_2,\;\;\;\;\text{for some}\;\;a_1>0,
\end{equation}
\begin{equation}
\label{cond3}
supp\,u(\cdot,1)\subset\{x\in\R^n\,:\,x_1\leq a_2\},\;\;\;\;\text{for some} \;\;a_2<\infty,
\end{equation}
with
\begin{equation}
\label{cond4}
 V\in L^{\infty}(\R^n\times [0,1]),\;\;\;\;\;\|V\|_{L^{\infty}(\R^n\times [0,1])}=M_0,
\end{equation}
and
\begin{equation}
\label{cond5}
\lim_{\rho\rightarrow +\infty}\|V\|_{L^1([0,1] : L^\infty(\R^n\setminus B_{\rho}))}=0.
\end{equation}
Then $\,u\equiv 0$.
\end{theorem}

\vskip.1in

\underline{Remarks:} (a) Note that in order to prove Theorem \ref{Theorem1}, 
by translation in $x_1$, we can choose who  $a_2$ is. We will show that there exists $m>0$ 
(small) with the property that if \eqref{cond1}, \eqref{cond2}, \eqref{cond4}, \eqref{cond5} hold and 
\eqref{cond3} holds with $a_2=m$, then
$$
u(x,1)=0\;\;\;\;\text{for}\;\;\;x\in\R^n\;\;\;\text{such that}\;\;\;m/2<x_1\leq m.
$$
This clearly yields the desired result. Without loss of generality we will assume $\,m<1$.

(b) By rescaling it is clear that the result in Theorem \ref{Theorem1} applies to any time interval $[0,T]$.

(c) We recall that in Theorem  \ref{Theorem1} there are no hypotheses on the size of the potential $V$ in the given class
or on its regularity.

(d) A weaker version of Theorem  \ref{Theorem1} was announced in \cite{EKPV12}.
 \vskip.07in
As a direct  consequence of  Theorem \ref{Theorem1}  we get the
following   result regarding the uniqueness of solutions for non-linear
equations of the form \eqref{E: NLS}.

\begin{theorem}
\label{Theorem 2}

Given 
$$
u_1,\,u_2 \in C([0,T] : H^k(\R^n)),\;\;\;\;\;\;\;\;0<T\leq \infty,
$$
strong solutions of \eqref{E: NLS}  with $k\in \Z^+$, $k>n/2$,
$F:\C^2\to \C$, $F\in C^{k}$  and $F(0)=\partial_uF(0)=\partial_{\bar
u}F(0)=0$
such that
\begin{equation}
\label{con2}
supp\,(u_1(\cdot,0)-u_2(\cdot,0))\subset \{x\in \R^n\,:\,x_1\leq a_2\},\;\;\;\;\;\;a_2<\infty.
\end{equation}
If for some $t\in(0,T)$ and for some $\epsilon>0$ 
\begin{equation}
\label{conditionp2b}
u_1(\cdot,t)-u_2(\cdot,t)\in L^2(e^{\epsilon\,|x_1|}\,dx),
\end{equation}
then $u_1\equiv u_2$.

\end{theorem}
\vskip.1in

\underline{Remarks:} (a) In particular, by taking $u_2\equiv 0$, Theorem \ref{Theorem 2} shows that if $u_1(\cdot,0)$ 
has compact support, then for any $t\in (0,T)$ 
$u_1(\cdot,t)$ cannot decay exponentially.

(b) In the case $F(u,\overline u)=|u|^{\alpha-1}u$, with $\alpha>n/2$ if $\alpha$ is not an odd integer, we have that if 
$\varphi$ is the unique non-negative, radially symmetric solution of
$$
-\Delta \varphi +\omega \,\varphi = |\varphi|^{\alpha-1}\varphi,\;\;\;\;\;\;\;\;\;\;\;\;\;\omega>0,
$$
then 
\begin{equation}
\label{rem1}
u_1(x,t)=e^{i\omega t}\varphi(x)
\end{equation}
 is a solution (\lq\lq standing wave") of 
\begin{equation}
\label{NLSre}
\partial_tu=i(\Delta u+|u|^{\alpha-1}u).
\end{equation}
It was established in \cite{Str}, \cite{BLi} that there exist constants $c_0,\,c_1>0$ such that
\begin{equation}
\label{rem2}
\varphi(x)\leq c_0\,e^{-c_1|x|}.
\end{equation}
Therefore, if we denote by $u_2(x,t)$ the solution of the equation \eqref{NLSre} with $\alpha>n/2$ and data $u_2(x,0)=\varphi(x)+\phi(x)$, $\phi\in H^{s}(\R^n),\,s>n/2$
having compact support it follows from Theorem \ref{Theorem 2}, \eqref{rem1} and \eqref{rem2} that for any $t\neq 0$ 
\begin{equation}
\label{007}
u_2(\cdot,t)\notin L^2(e^{\epsilon|x|}dx),\;\;\;\;\;\text{for any}\;\;\;\;\;\epsilon>0.
\end{equation}
In general, the same result \eqref{007} applies (in the time interval $[0,T]$) if one assumes  that  $u_1$ is a solution of \eqref{rem1} having exponential decay 
$$
|u_1(x,t)|\leq c_0\, e^{-c_1|x|},\;\;\;\;\;c_0,\,c_1>0\;\;\;\;\;(x,t)\in\R^n\times [0,T],
$$
and $u_2$ is the solution of \eqref{rem1} corresponding to an initial data 
$$
u_2(x,0)=u_1(x,0)+\phi(x),\;\;\;\;\;\phi\in H^{s}(\R^n), \;\;s>n/2\;\;\;\text{with compact support}.
$$

 The rest of this paper is organized as follows: section 2 contains all the preliminary results 
 to be used in the proof of Theorem \ref{Theorem1}. A  version of them has been proved in \cite{EKPV08m},
 \cite{EKPV09},
 \cite{EKPV12}. However, in some cases  modifications are needed to apply them in the setting considered here.
 Hence, some of their proofs will be sketched. 
 Section 3 contains the proof of Theorem \ref{Theorem1}.
\section{Preliminary Estimates}\label{section1}
In this section we describe the estimates to be used in the proof of Theorem \ref{Theorem1}. 

 First we recall a key step in the uniform exponential decay estimate established in  \cite{KPV02} :
 
  \begin{lemma}\label{Lemma1} There exists $\epsilon_n>0$ such that if
\begin{equation}
\label{hyp2}
\mathbb V:\mathbb R^n\times [0,1]\to\mathbb C,\;\;\;\;\text{with}\;\;\;\;
\|\mathbb V\|_{L^1_tL^{\infty}_x}\leq \epsilon_n,
\end{equation}
and $u\in C([0,1]:L^2(\mathbb R^n))$ is a strong solution of the IVP
\begin{equation}
\begin{cases}
\begin{aligned}
\label{eq1}
&\partial_tu=i(\Delta +\mathbb V(x,t))u+\mathbb G(x,t),\\
&u(x,0)=u_0(x),
\end{aligned}
\end{cases}
\end{equation}
with
\begin{equation}
\label{hyp3} u_0,\,u_1\equiv u(\,\cdot\,,1)\in
L^2(e^{2\lambda\cdot x}dx),\;\;\;\;\;\;\mathbb G\in L^1([0,1]:L^2(e^{2\lambda\cdot
x}dx)),
\end{equation}
for some $\lambda\in\mathbb R^n$, then there exists $c_n$ independent of
$\lambda$ such that
\begin{equation}
\begin{aligned}
\label{uno}
&\sup_{0\leq t\leq 1}\| e^{\lambda\cdot x} u(\,\cdot\,,t)\|_{L^2(\mathbb \R^n)} \\
&\leq c_n
\Big(\|e^{\lambda\cdot x} u_0\|_{L^2(\mathbb \R^n)} + \|e^{\lambda\cdot x}
u_1\|_{L^2(\mathbb \R^n)} +\int_0^1
\|e^{\lambda\cdot x}\, \mathbb G(\cdot, t)\|_{L^2(\mathbb \R^n)} dt\Big).
\end{aligned}
\end{equation}
\end{lemma}
\vspace{0,07 cm}

 Notice that in Lemma \ref{Lemma1} one assumes the existence of a reference $L^2$- solution $ u $ of the equation \eqref{eq1} 
 and gets  a control on the decay of the solution in the whole
time interval 
in terms of that  at the end  points and that of  the \lq\lq external force''. 
 In general, under appropriate assumptions on the potential $V(x,t)$ in \eqref{E:1.1} one writes
$$
V(x,t)u=  \chi_{\rho} V(x,t)u + (1-\chi_{\rho}) V(x,t)u = \mathbb V(x,t)u + \mathbb G(x,t),
$$
with $\chi_{\rho}\in C^{\infty}_0,\,$ $\chi_{\rho}(x)=1,\,|x|<\rho$, supported in $|x|<2\rho$,  and
obtains the estimate 
\eqref{uno} by fixing $\,\rho\,$ sufficiently large. Also under appropriate hypotheses
on $F$ and $u$  a similar argument 
can be used for 
the semi-linear equation in 
\eqref{E: NLS}.
\vskip.03in

Next, we recall the conformal or Appell transformation: 

  \begin{lemma}\label{lemma2}
If   $u(y,s)$ verifies
\begin{equation}
\label{2.1}
\partial_su=i\left(\triangle
u+V(y,s)u+F(y,s)\right),\;\;\;\;\;\;\;(y,s)\in \R^n\times [0,1],
\end{equation}
and $\alpha$ and $\beta$ are positive,  then
\begin{equation}
\label{2.2}
\widetilde u(x,t)=\left(\tfrac{\sqrt{\alpha\beta}}{\alpha(1-t)+\beta
t}\right)^{\frac n2}u\left(\tfrac{\sqrt{\alpha\beta}\,
x}{\alpha(1-t)+\beta t}, \tfrac{\beta t}{\alpha(1-t)+\beta
t}\right)e^{\frac{\left(\alpha-\beta\right) |x|^2}{4i(\alpha(1-t)+\beta
t)}},
\end{equation}
verifies
\begin{equation}
\label{2.3}
\partial_t\widetilde u=i\left(\triangle \widetilde u+\widetilde
V(x,t)\widetilde u+\widetilde F(x,t)\right),\;\;\;\;(x,t)\in  \R^n\times
[0,1],
\end{equation}
with
\begin{equation}
\label{potencial}
\widetilde V(x,t)=\tfrac{\alpha\beta}{\left(\alpha(1-t)+\beta
t\right)^2}\,V\left(\tfrac{\sqrt{\alpha\beta}\, x}{\alpha(1-t)+\beta t},
\tfrac{\beta t}{\alpha(1-t)+\beta t}\right),
\end{equation}
and
\begin{equation}
\label{externalforce}
\widetilde F(x,t)=\left(\tfrac{\sqrt{\alpha\beta}}{\alpha(1-t)+\beta
t}\right)^{\frac n2+2}F\left(\tfrac{\sqrt{\alpha\beta}\,
x}{\alpha(1-t)+\beta t}, \tfrac{\beta t}{\alpha(1-t)+\beta
t}\right)e^{\frac{\left(\alpha-\beta\right) |x|^2}{4i(\alpha(1-t)+\beta
t)}}.
\end{equation}
\end{lemma}

The following result  is a modified version of the one 
in \cite{EKPV06} (Lemma 3.1, page 1818). It will provide a needed lower bound  of the $L^2$-norm of the solution of the equation \eqref{E:1.1} 
and its first order derivatives in the $x_1$-variable  in  the domain $\{x\,:\,R-1<x_1<R\}\times [0,1]$.

\begin{lemma}\label{Lemma3} 
Assume that $ R>0$ large enough and that $\,\varphi : [0,1] \to \R$ is a smooth function. Then, there
exists 
$\,c=c(n;\|\varphi' \|_{\infty}+\|\varphi'' \|_{\infty})>0$ such that the inequality
\begin{equation}
\label{cpde1} 
\frac{\sigma^{3/2}}{R^2}\,\Big\|\,e^{\sigma |\frac{x_1-x_{0,1}}{R}+\varphi(t)|^2}g
\Big\|_{L^2(dxdt)}
\leq c\, \Big\|\,e^{\sigma |\frac{x_1-x_{0,1}}{R}+\varphi(t)|^2}(i
\partial_t+\Delta) g \Big\|_{L^2(dxdt)}
\end{equation}
holds when $\,\sigma \geq  c R^2 \,$ and $\,g\in C^{\infty}_0(\R^{n+1})\,$ is supported on
the set
$$
\{(x,t)=(x_1,..,x_n,t)\,\in\R^{n+1}\,:\, |\frac{x_1-x_{0,1}}{R}+\varphi(t)|\geq 1\}.
$$
\end{lemma}

 \begin{proof} As it was remarked above this result is a variation of the one  given in  detail in  \cite{EKPV06},
hence a sketch will suffice.

By translation, without loss of generality, we can assume $\,x_{0,1}=0$.  Let
$$
f(x,t)= e^{\sigma |\frac{x_1}{R}+\varphi(t)|^2}g(x,t).
$$
Then,
\begin{equation}
\label{3.3}
e^{\sigma|\frac {x_1}{R}+\varphi(t)|^2}(i \partial_t+\Delta)g=S_\sigma f-4\sigma A_\sigma f,
\end{equation}
where
\begin{align*}
&S_\sigma=i\partial_t+\Delta+\tfrac{4\sigma^2}{R^2}|\tfrac {x_1}{R}+\varphi|^2,\\
\\
&A_\sigma=\tfrac 1R\left (\tfrac {x_1}{R}+\varphi \right)\,\partial_{x_1}+ \tfrac 1{2R^2}+
\tfrac{i\ \varphi'}2\left(\tfrac {x_1}R+\varphi\right ).
\end{align*}
Thus,
\begin{equation}
\label{3.5}
S_\sigma^*=S_\sigma,\;\;\;\;\;\;\;A_\sigma^*=-A_\sigma,
\end{equation}
and integrating by parts (possible since $\,g\in C^{\infty}_0(\R^{n+1})\,$) one sees that
\begin{align*}
\|e^{\sigma |\frac{x_1}{R}+\varphi|^2}&(i \partial_t+\Delta)g\|_2^2
=\langle S_\sigma f-4\sigma A_\sigma f, S_\sigma f-4\sigma A_\sigma f\rangle\\
&\geq -4\sigma\langle (S_\sigma A_\sigma - A_\sigma S_\sigma )f, f \rangle =  -4\sigma\langle [S_\sigma ,A_\sigma]f, f\rangle\ .
\end{align*}
A calculation shows that
\begin{equation*}
 [S_\sigma ,A_\sigma]=\tfrac 2{R^2} \partial_{x_1}^2-\tfrac {4\sigma^2}{R^4}
|\tfrac {x_1}{R}+\varphi|^2-\tfrac 12 [(\tfrac
{x_1}R+\varphi)\varphi''+(\varphi')^2]+\tfrac {2i\varphi'}R\partial_{x_1}.
\end{equation*}

From this it follows that 
\begin{equation}
\label{3.7}
\begin{aligned}
&\|e^{\sigma|\tfrac {x_1}{R}+\varphi|^2}(i\partial_t+\Delta)g\|_2^2\\
&\;\;\;\;\geq  
\frac{16\sigma^3}{R^4} \int |\tfrac {x_1}{R}+\varphi|^2 |f|^2dxdt+\frac {8\sigma} {R^2}\int |\partial_{x_1}f|^2dxdt\\
&\;\;\;\;+2\sigma\int[(\tfrac {x_1}R+\varphi)\varphi''+(\varphi')^2]|f|^2dxdt- \Im\,(\frac{8\sigma i}R\int \varphi'\,\partial_{x_1} f\bar fdxdt)\ .
\end{aligned}
\end{equation}
 Now, when $\,\sigma\geq cR^2$ one has
 $$
 \frac{\sigma^3}{R^4}\geq c^2\sigma,
 $$
 so by taking $\,c$ large enough, depending on $\|\varphi' \|_{\infty} $ and $ \|\varphi'' \|_{\infty}$, and using that \newline
 $\,|\tfrac{x_1}{R}+\varphi(t)|\geq 1$ on the $supp (f)=supp (g)$, we can hide the third term on the right hand side (r.h.s.) in the inequality  \eqref{3.7} 
 in the first term in the r.h.s. 
 Also, since
 $$
 \aligned
& |\frac{8\sigma i}R\int |\varphi'|\,|\partial_{x_1} f| |\bar f| dxdt| \\
&\leq \frac{8\sigma}{R}\|\varphi'\|_{\infty}\int |f| |\partial_{x_1}f|
 \leq 4 \sigma \|\varphi'\|_{\infty}^2\int|f|^2dxdt+ \frac{4\sigma}{R^2}\int|\partial_{x_1}f|^2dxdt,
\endaligned
 $$
the contribution of this term in \eqref{3.7}  can be hidden by the first and second term in the r.h.s. of \eqref{3.7} if $c$ is large. This concludes the proof.
 \end{proof}
\vskip.05in

 Note that the same proof works by taking $c$ a bit larger, if  we only assume $ \,|\tfrac{x_1}{R}+\varphi(t)|\geq 1/2$ on $supp (g)$.

 In the proof of Theorem \ref{Theorem1} we shall  need the following extension of  Lemma \ref{Lemma3}.

\begin{corollary}\label{Corollary1} 

 Assume $g\in L^2(\R^{n+1})$ with $x_1,\,t$ on $\,supp(g)$  bounded,
$$
supp(g) \subset \{(x,t)=(x_1,..,x_n,t)\,\in\R^{n+1}\,:\, |\frac{x_1-x_{0,1}}{R}+\varphi(t)|\geq 1\}
$$
and $\,(i\partial_t+\Delta) g\in  L^2(\R^{n+1})$, then the inequality \eqref{cpde1} holds.
\end{corollary} 

\begin{proof} We can again assume that  $x_{0,1}=0$. We introduce the notation  $x=(x_1,x')\in\R\times \R^{n-1}$. Let $\eta_1\in C^{\infty}_0(\R)$,  $\eta_1\geq 0$, $\,supp(\eta_1)\subset\{|x_1|<1\}$ and
$\eta_2\in C^{\infty}_0(\R^{n-1})$, $\eta_2\geq 0$, $supp(\eta_2)\subset\{|x'|<1\}$ with
$$
\int_{\R}\,\eta_1(x_1)dx_1=1\;\;\;\;\;\;\text{and}\;\;\;\;\;\;\int_{\R^{n-1}} \eta_2(x')dx'=1.
$$

 For $\delta>0$ small define
$$
h_{\delta}(x,t)=\frac{1}{\delta^{n+2}} \eta_1(t/\delta^2) \eta_1(x_1/\delta) \eta_2(x'/\delta^{n-1})\;\;\;\;\;\;\text{and}\;\;\;\;\;\;\;g_{\delta}=h_{\delta}\ast g.
$$

 Let $\theta\in C^{\infty}_0(\R^{n-1}),\,\theta(x')=1,\;|x'|\leq 1$, and $\,supp(\theta)\subset \{|x'|<2\}$. For $l$ large, define
$$
g_{\delta,l}(x,t)=\theta(x'/l)\,g_{\delta}(x,t).
$$
Note that for $\delta>0$ small, 
$$
supp(g_{\delta})\subset \{(x,t)\,: \,|\tfrac{x_1}{R}+\varphi(t)|^2\geq 1/2\}, 
$$
and the same holds for $g_{\delta,l}$.
 Moreover, $g_{\delta,l}\in C^{\infty}_0(\R^{n+1})$.
 
 We apply Lemma \ref{Lemma3} to $g_{\delta,l}$ to obtain:
 \begin{equation}
 \label{paso1}
 \frac{\sigma^{3/2}}{R^2}\,\Big\|\,e^{\sigma |\frac{x_1}{R}+\varphi(t)|^2}g_{\delta,l}
\Big\|_{L^2(dxdt)}
\leq c\, \Big\|\,e^{\sigma |\frac{x_1}{R}+\varphi(t)|^2}(i
\partial_t+\Delta) g_{\delta,l} \Big\|_{L^2(dxdt)}.
\end{equation}
 Next, we fix $\delta>0$ small and see that
 \begin{equation}
 \label{paso2}
 \begin{aligned}
 (i\partial_t+\Delta) g_{\delta,l} &=\theta(x'/l)(i\partial_t+\Delta) g_{\delta}\\
& +\frac{2}{l}\nabla'\theta(x'/l)\cdot\nabla'g_{\delta}(x,t) +\frac{1}{l^2}\Delta\theta(x'/l)g_{\delta}(x,t).
\end{aligned}
\end{equation}
Therefore, by taking $\,l\to \infty$ the $L^2(dxdt)$-norm of the the last two terms on the r.h.s. of \eqref{paso2} tend to zero. Hence, inserting this in \eqref{paso1} we obtain the same estimate for $g_{\delta}$.
Next, we have that
$$
(i\partial_t+\Delta) g_{\delta}=(i\partial_t+\Delta) (h_{\delta}\ast g) =h_{\delta}\ast (i\partial_t+\Delta) g.
$$
 Using the supremum in $\delta$ (non-isotropic maximal function) and its boundedness, together with the boundedness of the support in $(x_1,t)$ of $g$, so that
 $$
 e^{\sigma|x_1/R+\varphi(t)|^2}\leq c_{\sigma, R},
 $$ 
by the dominated convergence theorem we can pass to the limit as $\delta \to 0$ to obtain the desired result.
\end{proof}

\section{Proof of Theorem \ref {Theorem1}}
\label{PW}
We divide our argument into six steps:
\vskip.05in
\underline{Step 1}: We claim that
\begin{equation}
\label{step1}
\sup_{0\leq t\leq 1}\int_{\R^n}\,e^{2a_1 x_1}|u(x,t)|^2dx\leq A_3.
\end{equation}

\underline{Proof of Step 1} : Using \eqref{cond5} in Theorem \ref{Theorem1} we choose $\rho$ so large such that 
$$
\| V\,\chi_{\{|x|\geq \rho\}}\,\|_{L^1([0,1]:L^{\infty}(\R^n))}\leq \epsilon_n,
$$
with $\epsilon_n$ as in Lemma \ref{Lemma1}. From \eqref{cond1}-\eqref{cond3} we have
$$
\int_{\R^n}\;e^{2a_1 x_1}|u(x,0)|^2dx\leq A_2,
$$
and
$$
\int_{\R^n}\;e^{2a_1 x_1}|u(x,1)|^2dx\leq A_1+e^{2a_1 m}\leq A_1+e^{2a_1}.
$$

We apply Lemma \ref{Lemma1}, with $\mathbb G(x,t)= - \chi_{\{|x|\leq \rho\}}\, V(x,t) u(x,t)$, using that
$$
\int_0^1\|  \,e^{a_1 x_1}\chi_{\{|x|\leq \rho\}}\, V\, u\|_2dt\leq  e^{a_1\rho}M_0 A_1,
$$
which gives step 1 with $A_3=A_3(A_1;A_2;a_1;M_0;\rho)$.
\vskip.1in

\underline{Step 2}: Define $\delta>0$ as
\begin{equation}
\label{defdelta}
\delta=\frac{\epsilon_n}{M_0+1},
\end{equation}
with $M_0$ as in \eqref{cond4} and $\epsilon_n$ as in Lemma \ref{Lemma1}. Note that $\delta<1$, and 
\begin{equation}
\label{step2a}
\int_{1-\delta}^1\| V(\cdot,t)\|_{\infty}dt\leq \epsilon_n.
\end{equation}
Let
\begin{equation}
\label{defv}
v(x,t)=u(\delta^{1/2}x,\delta t+1-\delta).
\end{equation}
We shall show that under the hypothesis of Theorem \ref{Theorem1} 
\begin{equation}
\label{goal2}
\int_{\tfrac{m}{2\delta^{1/2}}<x_1<\tfrac{m}{\delta^{1/2}}}|v(x,1)|^2dx=\int_{\tfrac{m}{2}<x_1<m}|u(x,1)|^2dx=0
\end{equation}
as desired.

Defining
\begin{equation}
\label{vdelta}
V_{\delta}(x,t)=\delta\, V(\delta^{1/2}x,\delta t +1-\delta)
\end{equation}
we see that $v(x,t)$ satisfies the equation
$$
\partial_t v=i(\Delta v+V_{\delta} v),\;\;\;\;\;\;(x,t)\in\R^n\times [0,1].
$$
We notice, using \eqref{step2a}, that
\begin{equation}
\label{est1}
\| V_{\delta}\|_{L^{\infty}(\R^n\times [0,1])}\leq M_0\,\delta\leq \epsilon_n,\;\;\;\;\;\;\;\int_0^1\| V_{\delta}(\cdot,t)\|_{\infty}dt\leq \epsilon_n,
\end{equation}
and 
$$
\int_{\R^n}|v(x,t)|^2dx=\frac{1}{\delta^{n/2}}\int_{\R^n}|u(y,\delta t+1-\delta)|^2dy\leq\frac{A_1}{\delta^{n/2}},
$$
with 
$$
supp (v(\cdot,1))\subset\{x_1\leq m/\delta^{1/2}\}.
$$
Thus, from \eqref{step1}
$$
\int_{\R^n} e^{2a_1x_1\delta^{1/2}}|v(x,0)|^2dx=\int_{\R^n}e^{2a_1x_1\delta^{1/2}}
|u(\delta^{1/2}x,1-\delta)|^2dx\leq\frac{A_3}{\delta^{n/2}}.
$$
We remark that $\delta$ was fixed in \eqref{defdelta} (independent of $m$), and that we can still choose $m $ small.

\vskip.1in
\underline{Step 3}: Using the Appell (conformal) transformation Lemma \ref{lemma2} we have that if 
$$
\partial_sv=i(\Delta v +V_{\delta} v),\;\;\;\;\;(y,s)\in\R^n\times [0,1],
$$
then for any $\alpha,\,\beta>0$ 
\begin{equation}
\label{defvtilde}
\widetilde v(x,t)=\left(\tfrac{\sqrt{\alpha\beta}}{\alpha(1-t)+\beta
t}\right)^{\frac n2}v\left(\tfrac{\sqrt{\alpha\beta}\,
x}{\alpha(1-t)+\beta t}, \tfrac{\beta t}{\alpha(1-t)+\beta
t}\right)e^{\frac{\left(\alpha-\beta\right) |x|^2}{4i(\alpha(1-t)+\beta
t)}},
\end{equation}
verifies
$$
\partial_t \widetilde v=i(\Delta \widetilde v+\widetilde V \,\widetilde v),\;\;\;\;(x,t)\in\R^n\times [0,1],
$$
with
$$
\widetilde V(x,t)=\frac{\alpha \beta}{(\alpha(1-t)+\beta t)^2}\,
V_{\delta}\Big(\frac{\sqrt{\alpha \beta} x}{\alpha(1-t)+\beta t},\frac{\beta t}{\alpha(1-t)+\beta t}\Big).
$$

For $\lambda>0$ given we will choose $\alpha=\alpha(\lambda,\delta),\;\beta=\beta(\lambda,\delta)$. We recall that
$$
\| e^{a_1 x_1\delta^{1/2}}\,v(\cdot,0)\|_2^2\leq \frac{A_3}{\delta^{n/2}},
$$
and from the support hypothesis 
$$
\| e^{\lambda x_1}\,v(\cdot,1)\|_2^2\leq \frac{ e^{2 m \lambda/\delta^{1/2}} A_1}{\delta^{n/2}}.
$$

We want  $\gamma=\gamma(\lambda,\delta)$ such that
$$
\| e^{\gamma x_1}\widetilde v(x,0)\|_2=\| e^{\gamma(\alpha/\beta)^{1/2} x_1} v(x,0)\|_2=\| e^{a_1x_1\delta^{1/2}}v(\cdot,0)\|_2\leq \frac{A_3^{1/2}}{\delta^{n/4}},
$$
and
$$
\| e^{\gamma x_1}\widetilde v(x,1)\|_2=\| e^{\gamma(\beta/\alpha)^{1/2} x_1} v(x,1)\|_2=\| e^{\lambda x_1}v(\cdot,1)\|_2\leq \frac{e^{\lambda m/\delta^{1/2}} A_1^{1/2}}{\delta^{n/4}}.
$$

Thus, we choose
\begin{equation}
\label{alphabeta}
\gamma(\alpha/\beta)^{1/2}=\delta^{1/2}a_1,\;\;\;\;\;\;\gamma(\beta/\alpha)^{1/2}=\lambda,
\end{equation}
i.e.
\begin{equation}
\label{alphabeta2}
\gamma=(\lambda\delta^{1/2}a_1)^{1/2},\;\;\;\;\beta=\lambda,\;\;\;\alpha=\delta^{1/2}a_1.
\end{equation}

Next, using the  change of variable 
$$
\widehat t= \frac{\beta}{\alpha(1-t)+\beta t},\;\;\;\;\;\;\;\;\;\;d \widehat t=\frac{\alpha \beta}{(\alpha(1-t)+\beta t)^2}\,dt,
$$
it follows that
$$
\aligned
&\int_0^1\|\,\widetilde V(\cdot,t)\|_{\infty}dt\\
\\
&=\int_0^1\|\frac{\alpha \beta}{(\alpha(1-t)+\beta t)^2}
V_{\delta}(\frac{\sqrt{\alpha \beta} x}{\alpha(1-t)+\beta t},\frac{\beta t}{\alpha(1-t)+\beta t})\|_{\infty}dt\\
\\
&=\int_0^1\|V_{\delta}(\cdot,\widehat t)\|_{\infty}d \widehat t\leq \epsilon_n,
\endaligned
$$
using \eqref{est1}. So we can apply Lemma \ref{Lemma1} again, this time with $\mathbb G\equiv 0$, to obtain that
\begin{equation}
\label{est2}
\begin{aligned}
\sup_{0\leq t \leq 1}\| e^{\gamma x_1} \widetilde v(\cdot,t)\|_2 &\leq c_n \Big(\frac{A_3^{1/2}}{\delta^{n/4}}+\frac{A_1^{1/2}}{\delta^{n/4}}\,e^{\lambda m/\delta^{1/2}}\Big)\\
&\leq c_{\delta,a_1,A_1,A_3}\,e^{\lambda m/\delta^{1/2}}\\
&\leq c\,e^{\lambda m/\delta^{1/2}},
\end{aligned}
\end{equation}
if $\lambda >0$ is large and $A_1\neq 0$, (how large $\lambda $ is for this depends on $\,m, A_1, A_3$ and $\delta$, but this will not matter).
 Note  that
\begin{equation}
\label{uno1}
\| \widetilde v(\cdot,t)\|_2^2=\| v(\cdot,\frac{\beta t}{\alpha(1-t)+\beta t})\|_2^2\leq \frac{A_1}{\delta^{n/2}},
\end{equation}
hence
\begin{equation}
\label{est3}
\sup_{0\leq t\leq 1}\| \widetilde v(\cdot,t)\|_2\leq \frac{A_1^{1/2}}{\delta^{n/4}}.
\end{equation}

Now, we denote by  $\phi_0(x_1)\geq 0$  a $C^{\infty}$ convex function
such that
\begin{equation}
\label{weight}
\phi_0(x_1)=
\begin{cases}
\begin{aligned}
&0,\;\;\;\;\; \;\;\;\;\;\;\;\;\;\;x_1\leq 0,\\
&x_1-1/4,\;\;\;x_1\geq 1/2,
\end{aligned}
\end{cases}
\end{equation}
and define
$$
\phi(x_1)=(1+(\phi(x_1))^2)^{1/2}.
$$
Since $\gamma=(\lambda\delta^{1/2} a_1)^{1/2}$, from \eqref{est2} and large $\lambda$ we have
\begin{equation}
\label{est4}
\sup_{0\leq t\leq 1}\| e^{\gamma \phi(x_1)} \widetilde v(\cdot,t)\|_2\leq c_{\delta,a_1}\,e^{\lambda m/\delta^{1/2}}.
\end{equation}

A computation shows that
$$
\phi'(x_1)=\frac{\phi_0(x_1)\, \phi'_0(x_1)}{(1+(\phi_0(x_1))^2)^{1/2}},
$$
and 
$$
\phi^{''}(x_1)=\frac{(\phi'_0(x_1))^2}{(1+(\phi_0(x_1))^2)^{3/2}} + \frac{\phi_0(x_1)\, \phi^{''}_0(x_1)}{(1+(\phi_0(x_1))^2)^{3/2}}.
$$
Thus, for $x_1\geq 1/2$ one has that
\begin{equation}
\label{3stars}
\phi^{''}(x_1)\geq \frac{1}{4} \frac{1}{(1+x_1^2)^{3/2}} =\frac{1}{4}\frac{1}{\langle x_1\rangle^3}.
\end{equation}

We now follow an argument similar to that in \cite{EKPV08m} section 2.  Let
$$
f(x,t)=e^{\gamma \phi(x_1)}\,\widetilde v(x,t).
$$
Then $f$ verifies
\begin{equation}\label{E: loquefcumple}
\partial_tf=\mathcal Sf+\mathcal Af +i\,e^{\gamma\phi}F,\;\;\;\;
\text{in}\;\;\;\R^n\times[0,1],
\end{equation}
with symmetric and skew-symmetric operators $\mathcal S$ and $\mathcal A$
\begin{equation}
\label{E: formulaoperadores}
\begin{aligned}
\mathcal
S=&-i\gamma\left(2\partial_{x_1}\phi\,\partial_{x_1}+\partial_{x_1}^2\phi\right),\\
\mathcal A= &i\left(\triangle +\gamma^2|\partial_{x_1}\phi|^2\right).
\end{aligned}
\end{equation}
and
$$
F=\widetilde V\,\widetilde v.
$$
A calculation shows that,
\begin{equation}\label{E: formulaconmutadorindependientetiempo}
\mathcal S_t+\left[\mathcal S,\mathcal A\right]=- \gamma
\left[4\partial_{x_1}\,\phi^{''}\partial_{x_1}-4
\gamma^2\phi^{''}(\phi')^2+\phi^{(4)}\right].
\end{equation}

By Lemma 2 in \cite{EKPV08m}
\begin{equation}
\label{E: derivadasegunda}
\begin{aligned}
\partial_t^2H\equiv \partial_t^2\left( f, f\right)=& 2\partial_t\text{\it
Re}\left(\partial_tf-\mathcal Sf-\mathcal Af,f\right)+2\left(\mathcal
S_tf+\left[\mathcal S,\mathcal A\right]f,f\right)\\
&+\|\partial_tf-\mathcal Af+\mathcal Sf\|^2-\|\partial_tf-\mathcal
Af-\mathcal Sf\|^2,
\end{aligned}
\end{equation}
so
\begin{equation}
\label{E-b}
\begin{aligned}
\partial_t^2H& \geq  2\partial_t \text{\it Re}\left(\partial_tf-\mathcal
Sf-\mathcal Af,f\right)\\
&+2\left(\mathcal S_tf+\left[\mathcal S,\mathcal
A\right]f,f\right)-\|\partial_tf-\mathcal Af-\mathcal Sf\|^2.
\end{aligned}
\end{equation}

Multiplying \eqref{E-b} by $t(1-t)$ and integrating in $t$ we obtain
\begin{equation}
\label{EH}
\begin{aligned}
&2\int_0^1t(1-t)\left(\mathcal S_tf+\left[\mathcal S,\mathcal
A\right]f,f\right)dt\\
&\leq c_n\,\sup_{[0,1]}\|e^{\gamma\,\phi}
\,\widetilde v(t)\|_2^2+c_n\,\sup_{[0,1]}\|e^{\gamma\,\phi}
F(t)||_2.
\end{aligned}
\end{equation}
This computation can be justified by parabolic regularization using the
fact that we already know the decay estimate for $\,\widetilde v$, see \cite{EKPV08b}. 
Note that for $\lambda$ sufficiently large
\begin{equation}
\label{NN}
\| \widetilde V\|_{\infty}\leq \Big(\frac{\beta}{\alpha}\Big)\,\|V_{\delta}\|_{\infty}\leq \frac{\lambda}{\delta^{1/2}a_1}\,\delta M_0=\frac{\lambda \delta^{1/2} M_0}{a_1}.
\end{equation}

Hence, combining \eqref{est2}, \eqref{E: formulaconmutadorindependientetiempo}, and \eqref{NN} it follows that
\begin{equation}
\label{aa1}
\begin{aligned}
&8 \,\gamma \,\int_0^1\int \,t(1-t)\phi^{''}(x_1)\, |\partial_{x_1}f|^2  dx dt \\
\\
&+
8\, \gamma^3 \,\int_0^1\int \,t(1-t) \,\phi^{''}(x_1)\,(\phi'(x_1))^2 \,|f|^2dxdt\\
\\
&\leq c_n\,\gamma\,\sup_{[0,1]}\|f(\cdot,t)\|_2^2
+ c_{\delta,M_0,a_1,n}\,\lambda\,\sup_{[0,1]}\|f(\cdot,t)\|_2^2
+c_n\,\sup_{[0,1]}\|f(\cdot,t)\|_2^2\\
\\
&\leq c_{\delta,M_0,a_1,n}\,\lambda\,e^{2\lambda m/\delta^{1/2}}.
\end{aligned}
\end{equation}

We recall that
$$
\partial_{x_1} f= e^{\gamma\phi(x_1)}\,\partial_{x_1}\widetilde v+
\gamma\,e^{ \gamma\phi(x_1)}\,\phi'(x_1)\,\widetilde v,
$$
thus
$$
\aligned
&\gamma |\partial_{x_1}f|^2=\gamma e^{2\gamma\phi(x_1)}|\partial_{x_1}\widetilde v+\gamma \phi'(x_1)\widetilde v|^2\\
&=e^{2\gamma\phi(x_1)}(\gamma |\partial_{x_1}\widetilde v|^2 +2\gamma^2\phi'(x_1)\widetilde v\, \partial_{x_1}\widetilde v + \gamma^3(\phi'(x_1))^2|\widetilde v|^2),
\endaligned
$$
with 
$$
| 2\gamma^2\phi'(x_1)\widetilde v\, \partial_{x_1}\widetilde v|\leq \frac{1}{2}\gamma |\partial_{x_1}\widetilde v|^2 +2\gamma^3(\phi'(x_1))^2|\widetilde v|^2.
$$
Inserting these estimates in \eqref{aa1} for  $\lambda$ large one gets 
$$
4 \gamma\int_0^1\int\,t(1-t)\,\phi^{''}(x_1)
\,e^{2\gamma \phi(x_1)}\,|\partial_{x_1}\widetilde v|^2 dx dt
\leq c_{\delta,M_0,a_1,n}\,\lambda\,e^{2\lambda m/\delta^{1/2}}.
$$
Hence, for $x_1>1/2$ from \eqref{3stars} one has that
$$
 \gamma\int_0^1\int\,t(1-t)\frac{1}{\langle x_1\rangle^3}
\,e^{2\gamma \phi(x_1)}\,|\partial_{x_1}\widetilde v|^2 dx dt
\leq c_{\delta,M_0,a_1,n}\,\lambda\,e^{2\lambda m/\delta^{1/2}},
$$
for $\lambda$ large. Collecting the above information, \eqref{est2}, and \eqref{aa1}  we conclude that
\begin{equation}
\label{conclu1}
\begin{aligned}
&\sup_{0\leq t\leq 1}\| e^{\gamma \phi(x_1)} \widetilde v(\cdot,t)\|^2_2
+\gamma\int_0^1\int_{x_1>\tfrac{1}{2}}\,t(1-t)\frac{1}{\langle x_1\rangle^3}
\,e^{2\gamma \phi(x_1)}\,|\partial_{x_1}\widetilde v|^2 dx dt\\
&\leq c_{\delta,M_0,a_1,n}\,\lambda\,e^{2\lambda m/\delta^{1/2}}.
\end{aligned}
\end{equation}
\vskip.1in

\underline{Step 4 } : We will give lower bounds for
$$
\Phi=\int_{2\leq x_1\leq R/2}\,\int_{3/8}^{5/8} \,| \widetilde v(x,t)|^2dt dx,
$$
for $R$ large to be chosen.

First, we recall that
$$
\Phi=\int_{2\leq x_1\leq R/2}\,\int_{3/8}^{5/8} 
\Big|\left(\tfrac{\sqrt{\alpha\beta}}{\alpha(1-t)+\beta
t}\right)^{\frac n2}v\left(\tfrac{\sqrt{\alpha\beta}\,
x}{\alpha(1-t)+\beta t}, \tfrac{\beta t}{\alpha(1-t)+\beta
t}\right)\Big|^2dxdt.
$$
 Next, for $t\in[3/8,5/8]$ we see that
 $$
 s(t)=\frac{\beta\,t}{\alpha(1-t)+\beta t},
 $$
 satisfies that
 $$
 dt=\frac{(\alpha(1-t)+\beta t)^2}{\alpha \beta} ds\simeq \frac{\beta^2}{\alpha \beta}\, ds=\frac{\beta}{\alpha}\,ds,
 $$
 with 
 $$
 s(3/8)=\frac{3\beta}{5\alpha+3\beta}\in(1/2,1),
 $$
 and
 $$
 s(5/8)=\frac{5\beta}{3\alpha+5\beta}\in(1/2,1).
 $$
Therefore
$$
s(5/8)-s(3/8)=\frac{2\alpha \beta}{(5\alpha+3\beta)(3\alpha+5\beta)}\,\cong \frac{\alpha}{\beta},
$$
for large $\lambda$, and
$$
s(5/8)>s(3/8)\uparrow 1\;\; \;\;\;\text{as}\;\;\;\lambda \uparrow \infty.
$$

In the $x$-variable we have
$$
y=\frac{\sqrt{\alpha\beta}}{\alpha(1-t)+\beta}\,x,
$$
 so for $t\in[3/8, 5/8]$ and $2<x_1<R/2$ one basically has that
 $$
 y_1\in [2\sqrt{\tfrac{\alpha}{\beta}}, \tfrac{R}{2} \sqrt{\tfrac{\alpha}{\beta}}\,]\equiv A.
 $$
 
Thus,
 \begin{equation}
 \label{4stars}
 \Phi\geq c_n\,\frac{\beta}{\alpha}\,\int_{A}\int_{I_{\lambda}}|v(y,s)|^2ds dy,
\end{equation}
 with 
 $$I_{\lambda}=[s(3/8),s(5/8)],\;\;\;\;\;|I_{\lambda}|\cong \frac{\alpha}{\beta}\;\;\;\;\;\text{for}\;\;\;\;\lambda>>1,
 $$
 with 
 $$
 s(3/8)\to 1\;\;\;\text{as}\;\;\;\lambda\uparrow \infty.
 $$
 We choose
 \begin{equation}
 \label{choiceR}
 R=\frac{2 M \lambda^{1/2} m}{(\delta^{1/2}a_1)^{1/2}c_n},
 \end{equation}
 with
 \begin{equation}
 \label{choiceM}
\frac{2}{c_n} M\geq \frac{1}{\delta^{1/2}},
 \end{equation}
 to be fixed latter. Since
 $$
2 \sqrt{\frac{\alpha}{\beta}}=2\frac{\delta^{1/4}a_1^{1/2}}{\lambda^{1/2}}\to 0\;\;\;\;\text{as}\;\;\; \lambda\uparrow \infty,
 $$
 and
 $$
 \frac{R}{2}\sqrt{\frac{\alpha}{\beta}}=\frac{2M\,m}{c_n}\geq \frac{m}{c_n\delta^{1/2}}.
 $$
 Hence, from \eqref{4stars} we can conclude
 \begin{equation}
 \label{conclustep4}
\liminf_{\lambda\uparrow \infty} \Phi\geq  c_n\,\int_{0<y_1<\tfrac{m}{\delta^{1/2}}}\,|v(y,1)|^2dy. \end{equation}

\vskip.1in

\underline{Step 5} : Upper bounds for 
$$
\Xi(R)\equiv \int_{\tfrac{1}{2}<x_1<R}\;\int_{1/32}^{31/32}(|\widetilde v(x,t)|^2+
|\partial_{x_1} \widetilde v(x,t)|^2)dtdx.
$$

For the square of the $L^2$-norm of $\widetilde v$ we have the bound 
$A_1/\delta^{n/2}$, see \eqref{uno1}.
 For the square of the $L^2$-norm of $\partial_{x_1}\widetilde v$ using the conclusion of Step 3 \eqref{conclu1} we get the upper bound
 $$
 c_{\delta,M_0,a_1,n}\,(1+R^3)\,
 \lambda\,e^{2\lambda m/\delta^{1/2}}.
 $$
\vskip.1in

\underline{Step 6} : Carleman estimate \cite{Car} and conclusion of the proof.
\vskip.05in
We assume that for $m>0$ to be chosen
\begin{equation}
\label{defib}
b\equiv \int_{\tfrac{m}{2\delta^{1/2}}<y_1<\tfrac{m}{\delta^{1/2}}}\,|v(y,1)|^2dy>0.
 \end{equation}

We recall that 
$$
supp (v(\cdot,1))\subset \{y_1<m/\delta^{1/2}\}.
$$

From step 4 we have that for $\lambda$ sufficiently large
\begin{equation}
\label{CEK}
\int_{2\leq x_1\leq R/2}\;\int_{3/8}^{5/8}|\widetilde v(x,t)|^2dt dx\geq \frac{b}{2}.
\end{equation}

Now, let
$$
x_{0,1}=R/2,
$$
and $\varphi :[0,1]\to \R$ be a smooth function such that $0\leq \varphi(t)\leq 3/2-1/R$,
\begin{equation}
\label{defverphi}
\varphi(t)=
\begin{cases}
\begin{aligned}
&3/2-1/R,& t\in[3/8,5/8],\\
&0,&t\in [0,1/4]\cup[3/4,1],
\end{aligned}
\end{cases}
\end{equation}
with $\varphi,\,\varphi', \,\varphi^{''}$ uniformly bounded in $R$ for $R$ large.
 We fix 
 $$
 \sigma=cR^2,
 $$
 with $c$ denoting a universal constant whose value may change from line to line,
 so that Corollary \ref{Corollary1} applies. Chose $\theta_R\in C^{\infty}(\R)$, with 
 $0\leq \theta(x_1)\leq 1$ and
\begin{equation}
\label{defthetai1}
\theta_R(x_1)=
\begin{cases}
\begin{aligned}
&1,\;\;\;\;\;1<x_1<R-1,\\
&0,\;\;\;\;\,x_1<1/2\;\;\;\text{or}\;\;\;x_1>R.
\end{aligned}
\end{cases}
\end{equation} 
 Let $\zeta\in C^{\infty}(\R)$ satisfy 
 $0\leq \zeta(x_1)\leq 1$ and
\begin{equation}
\label{defthetai}
\zeta(x_1)=
\begin{cases}
\begin{aligned}
&0,\;\;\;\;x_1<1,\\
&1,\;\;\;x_1>1+1/(2R).
\end{aligned}
\end{cases}
\end{equation} 
Define
\begin{equation}
\label{defg}
g(x,t)\equiv \theta_R(x_1)\,\zeta\Big(\frac{x_1-R/2}{R}+\varphi(t)\Big)\,\widetilde v(x,t).
\end{equation}

Let us see that $g(x,t)$ verifies the hypotheses of Corollary \ref{Corollary1} so we can apply 
the inequality \eqref{cpde1}. First, it is clear that  it is supported on the set
$$
1/2<x_1<R,\;\;\;\;\;\;\;\;1/32<t<31/32,\;\;\;\;\;\;\;\;\;\;\Big|\frac{x_1-R/2}{R}+\varphi(t)\Big|\geq 1.
$$
 Below we shall see that
$$
(i\partial_t+\Delta)g\in L^2(dxdt).
$$

Note that 
\begin{equation}
\label{key1}
\text{if}\;\;\;\;3/2\leq x_1\leq R-1\;\;\;\;\;\text{and}\;\;\;\;\;3/8\leq t\leq 5/8,\;\text{then}\;\;\;\;g(x,t)=\widetilde v(x,t).
\end{equation}
In this domain, $\theta_R(x_1)\equiv 1$, and 
$$
\frac{x_1-R/2}{R}+\varphi(t)=\frac{x_1}{R}+1-\frac{1}{R}\geq 1+\frac{1}{2R},
$$
which gives \eqref{key1}.

 Also if $x_1>2$ one has $x_1/R+1-1/R\geq 1+1/R$, so that we have a
 lower bound $\Gamma$ for the left hand side of \eqref{cpde1} squared  with
 \begin{equation}
 \label{leftside}
\Gamma\equiv  \frac{\sigma^3}{R^4}\,e^{2\sigma(1+1/R)^2}\,\int_{2<x_1<R-1}\,\int_{3/8}^{5/8}|\widetilde v(x,t)|^2dtdx\geq \frac{b}{2}c^3R^2\,e^{2\sigma(1+1/R)^2},
 \end{equation}
 for $R$ large from \eqref{CEK}. The equation for $g$ is
 \begin{equation}
 \label{eqforg}
 \begin{aligned}
 &(i\partial_t+\Delta)g=\theta_R(x_1)\zeta\Big(\frac{x_1-R/2}{R}+\varphi(t)\Big)
 \widetilde V(x,t)\,\widetilde v\\
 &+\Big[\zeta\Big(\frac{x_1-R/2}{R}+\varphi(t)\Big)(2\theta_R'(x_1)\partial_{x_1}\widetilde v+\widetilde v\,\theta_R^{''}(x_1))\Big]\\
& +\Big[(i\zeta'(\cdot)\varphi'(t)
 +\zeta^{''}(\cdot)\frac{1}{R^2})\theta_R(x_1)\widetilde v+\frac{2}{R}\zeta'(\cdot)\theta_R(x_1)\partial_{x_1}\widetilde v\Big]\\
 &\equiv E_1+E_2+E_3.
 \end{aligned}
 \end{equation}
 
  Note that $\,(i\partial_t+\Delta)g\in L^2(dxdt)$ by Step 5.

 On the domain $1/2<x_1<R,\;1/32<t<31/32$ (which contains the support of $g$) one has  (see \eqref{alphabeta2})  
 $$
 \|\widetilde V\|_{\infty}\lesssim \frac{\alpha}{\beta}\, M_0\;\;\;\;\;\text{for}\;\;\;\;\lambda>>1.
 $$
 Thus, since
 $$
 \frac{\sigma^3}{R^4}=c^3R^2,
 $$
 for $R$ sufficiently large we can absorb the contribution of the term containing $E_1$ in the right hand side of \eqref{cpde1} in the left hand side of \eqref{cpde1}. So we have
 $$
 \aligned
 \frac{b}{2}\,c^3R^2e^{2\sigma(1+1/R)^2}&\leq c\int\int\,|E_2|^2\,e^{2\sigma|\frac{x_1-R/2}{R}+\varphi(t)|^2}dxdt\\
 &+c\int\int\,|E_3|^2\,e^{2\sigma|\frac{x_1-R/2}{R}+\varphi(t)|^2}dxdt.
\endaligned
 $$
  Next, we analyze the contribution of $E_2$. In this case, each term contains a factor equal to a derivative of  $\theta_R$,
  so the possible contribution are from the sets : \newline $1/2<x_1<1$ and $R-1<x_1<R$.
 If $1/2<x_1<1$, then 
 $$
 \frac{x_1-R/2}{R}+\varphi(t)\leq \frac{1}{R}-\frac{1}{2}+\frac{3}{2}-\frac{1}{R}=1,
 $$
 so in this domain
 $$
 \zeta\big( \frac{x_1-R/2}{R}+\varphi(t)\Big)\equiv 0.
 $$
   In the region $R-1<x_1<R$ we have
   $$
  \frac{x_1-R/2}{R}+\varphi(t)\leq1 -\frac{1}{2}+\frac{3}{2}-\frac{1}{R}=2-\frac{1}{R},
  $$
  so the contribution of the term involving $E_2$ is bounded above by
   $$
   \aligned
   &\int^{31/32}_{1/32}\int_{R-1<x_1<R}(|\widetilde v|^2+|\partial_{x_1}\widetilde v|^2)\,e^{2\sigma(1-1/R)^2}dxdt\\
   &=e^{2\sigma(2-1/R)^2}\;\int^{31/32}_{1/32}\int_{R-1<x_1<R}(|\widetilde v|^2+|\partial_{x_1}\widetilde v|^2)dxdt\\
   &=\Xi(R)\,e^{2\sigma(2-1/R)^2}.
   \endaligned
   $$
 
 Next, we consider the term involving $E_3$. In this case,  each term contains a factor equal to a  derivative of   $\zeta$ so its support is restricted to
 $$
 1\leq  \frac{x_1-R/2}{R}+\varphi(t)\leq 1+1/2R,
 $$
 with $1/2<x_1<R$ (support of $\theta_R$) and $t\in(1/32, 31/32)$. Hence, its contribution is bounded by (see \eqref{uno1})
 $$
\aligned
& c  R^4 \int_{1/32}^{31/32}\int_{\frac{1}{2}<x_1<R}|\widetilde v(x,t)|^2 e^{2\sigma(1+1/(2R))^2}dxdt\\
 &\leq c\,R^4\,e^{2\sigma(1+1/(2R))^2}  \int_{1/32}^{31/32}\int_{\frac{1}{2}<x_1<R}|\widetilde v(x,t)|^2
 dxdt \\
 &\leq c\,e^{2\sigma(1+1/(2R))^2} \;\frac{A_1}{\delta^{n/2}}\,R^4.
 \endaligned
 $$
 
 Collecting this information and using that $R$ is large we get
 \begin{equation}
 \label{aaa1}
\frac{b}{2} c^3R^2e^{2\sigma(1+1/R)^2}\leq c\,\Xi(R)e^{2\sigma(2-1/R)^2} +c_{A_1,\delta}e^{2\sigma(1+1/2R)^2} R^4.
\end{equation}
 Since $R$ is large the second term on the right hand side \eqref{aaa1} can be hidden on the left to get that
  \begin{equation}
 \label{aa2}
\frac{b}{4} c^3R^2e^{2\sigma(1+1/R)^2}\leq c\,\Xi(R)e^{2\sigma(2-1/R)^2}.
\end{equation}
 Now, since $\sigma=cR^2$ one has that
 $$
 2\sigma(1+\tfrac{1}{R})^2-2\sigma(2-\tfrac{1}{R})^2= - 6cR^2+12cR\geq -10 c R^2,
 $$
for $R$ large. Thus, from \eqref{aa2} it follows that 
\begin{equation}
\label{final1}
\frac{b}{4} c^3R^2e^{-10cR^2}\leq c\,\Xi(R).
\end{equation}
But using \eqref{conclu1}
\begin{equation}
\label{final2}
\begin{aligned}
\Xi(R)&=\int_{R-1<x_1<R}\,\int^{31/32}_{1/32}(|\widetilde v|^2+|\partial_{x_1}\widetilde v|^2)(x,t)dtdx\\
&= \int_{R-1<x_1<R}\,\int^{31/32}_{1/32}\,e^{2\gamma\phi(x_1)}\,
e^{-2\gamma\,\phi(x_1)}(|\widetilde v|^2+|\partial_{x_1}\widetilde v|^2)(x,t)dtdx\\
&\leq ce^{-\gamma R}R^3\lambda\,e^{2\lambda m/\delta^{1/2}}
\leq c R^3  e^{-\gamma R} e^{2\lambda m/\delta^{1/2}},
\end{aligned}
\end{equation}
for $\lambda >>1$ and $x_1>R>>1$ one has that
$$
\frac{1}{2} x_1\leq \phi(x_1)\leq x_1.
$$
 Thus, inserting \eqref{final2} into \eqref{final1} it follows that
\begin{equation}
\label{final3}
b\leq c\,e^{10cR^2-\gamma R+3\lambda m/\delta^{1/2}},
\end{equation}
where
$$
R=\frac{2M\lambda^{1/2}m}{(\delta^{1/2}a_1)^{1/2}c},\;\;\;\;\;\gamma=(\lambda\delta^{1/2}a_1)^{1/2},\;\;\;\;\;\;\;\text{and}\;\;\;\;\;\frac{2}{c_n}M\geq 1/\delta^{1/2}.
$$
So we have, changing $c_n$ into  $c$,
$$
\aligned
&10cR^2-\gamma R+3\lambda m/\delta^{1/2}\\
&=\frac{40cM^2m^2\lambda}{c^2\delta^{1/2}a_1}-(\lambda\delta^{1/2}a_1)^{1/2}
\frac{2M\lambda^{1/2}m}{(\delta^{1/2}a_1)^{1/2}c} +3\frac{\lambda m}{\delta^{1/2}}\\
&=\lambda\Big(\frac{40cM^2m^2}{c^2\delta^{1/2}a_1}-\frac{2M m}{c}+\frac{3m}{\delta^{1/2}}\Big).
\endaligned
$$
We need the expression in parenthesis to be negative, i.e.
$$
\frac{40cM^2m^2}{c^2\delta^{1/2}a_1}+\frac{3m}{\delta^{1/2}}< \frac{2M m}{c}.
$$
Divide by $\,M m\,$, we need 
$$
\frac{40 c M m}{c^2 \delta^{1/2} a_1} +\frac{3}{M \delta^{1/2}} <  \frac{2}{c}.
$$
 First, we choose $M$ so large such that $2M/c\geq 1/\delta^{1/2}$ and
 $$
 \frac{3}{M\delta^{1/2}}\leq \frac{1}{c}.
 $$
 So now we just need
 $$
 \frac{40cM m}{c^2\delta^{1/2}a_1}<\frac{1}{c}.
 $$
 
 This can be done by taking $m>0$ small.
  Therefore, we have proved that $b=0$, which yields the desired result.
  


\end{document}